\documentclass{amsart}

\usepackage{amssymb,latexsym,amscd,amsmath, amsthm, bold-extra, enumitem, hyperref, dsfont, mathtools}

\usepackage{color}

\numberwithin{equation}{section}

\long\def\eatit#1{}

\newtheorem{Thm}{Theorem}[section]
\newtheorem{Prop}[Thm]{Proposition}
\newtheorem{Lem}[Thm]{Lemma}

\theoremstyle{definition}
\newtheorem{Def}[Thm]{Definition}
\newtheorem{Ex}[Thm]{Example}
\newtheorem{Rmk}[Thm]{Remark}

\newcommand{\PP}{{\mathbb{P}}}
\newcommand{\CC}{{\mathbb{C}}}

\newcommand{\hadamardot}{\star}

\begin{document}


\title{Hadamard products of hypersurfaces}

\author{Cristiano Bocci, Enrico Carlini}

\address{Cristiano Bocci\\
Department of Information Engineering and Mathematics, University of Siena\\
Via Roma, 56 Siena, Italy}
\email{cristiano.bocci@unisi.it}

\address{Enrico Carlini\\
 Department of Mathematical Sciences, Politecnico di Torino\\
Corso Duca degli Abruzzi 24, Turin, Italy}
\email{enrico.carlini@polito.it}


\begin{abstract} In this paper we characterize hypersurfaces for which their Hadamard product is still a hypersurface
Then we study hypersurfaces and, more generally, varieties which are indempotent under Hadamard powers.

\end{abstract}


\maketitle

\section{Introduction}

According to the definition in \cite{CMS, CTY}, the Hadamard product between projective varieties $V,W\subset\PP^n$,  is the closure of the image of the rational map
\[
V \times W \dashrightarrow \PP^n, \quad  ([a_0:\dots : a_n], [b_0: \dots :b_n])\mapsto [a_0b_0 :a_1b_1:\cdots :a_nb_n].
\]
For any projective variety $V$, we may consider its Hadamard square $V^{\hadamardot 2} = V \hadamardot V$ and its higher Hadamard powers $V^{\hadamardot r} = V \hadamardot V^{\hadamardot (r-1)}$.

In \cite{CMS}, the authors use this definition to describe the algebraic variety associated to the restricted Boltzmann machine, which is the undirected graphical model for binary random variables specified by the bipartite graph $K_{r,n}$. This variety is the $r-$th Hadamard power of the first secant variety of $(\PP^1)^n$. Note that  \cite{CTY} concerns the case $r=2, n=4$.

Hadamard products and powers are  well-connected to other operations of varieties.  They are the multiplicative analogs of joins and secant varieties, and in tropical geometry, tropicalized Hadamard products equal Minkowski sums.  It is then natural to study properties of this new operation, and see its effects on various varieties.  

For this reasons, in the last few years, the Hadamard product of the projective varieties has been widely studied from the point of view of Projective Geometry. The main problem in this setting is the behaviour of the Hadamard product between varieties with many zero.
The paper \cite{BCK}, where the Hadamard products of linear space are studied,  can be considered the first step in this direction. Successively, the first author, with Calussi, Fatabbi and Lorenzini (\cite{BCFL1})   address the Hadamard product of linear varieties not necessarily in general position, obtaining, in $\mathbb{P}^2$ a complete description of the possible outcomes. In $\mathbb{P}^3$, under suitable conditions (which can be prove to be generic), can be shown that, for two set $V$ and $V'$ of collinear points, $V\star V'$ consists of $| V||V'|$ points on the two different rulings of a non-degenerate quadric.
Then, in \cite{BCFL2}, they address the Hadamard product of  not necessarily generic linear varieties and show that the Hilbert function of the Hadamard product $V\star W$ of two  varieties, with $\dim(V), \dim(W)\leq 1$, is the product of the Hilbert functions of the original varieties $V$ and $W$.   Moreover, they show that the  Hadamard
product of two generic linear varieties $V$ and $W$ is projectively equivalent to a Segre embedding.
In \cite{CCFL},  the second author, with Calussi, Fatabbi and Lorenzini consider generic degenerate subvarieties $V_i\subset \PP^n$ and  compute dimension and degree formulas for the Hadamard product of the varieties $V_i$.
 
The construction of star configurations of points, via Hadamard product, described in \cite{BCK}, found a generalization in \cite{CCGVT} where the authors introduce a new construction using the Hadamard product to present star configurations of codimension $c$ of $\PP^n$ and which they called Hadamard star configurations. Successively,  Bahmani Jafarloo and Calussi introduce a more general type of Hadamard star configuration; any star configuration constructed by their approach is called a weak Hadamard star configuration. In \cite{BJC} they classify weak Hadamard star configurations and, in the case $c=n$, they investigate the existence of a (weak) Hadamard star configuration which is apolar to the generic homogeneous polynomials of degree $d$.

In \cite{FOW}, the Hadamard product is studied for projective varieties and, in particular, the authors consider Hadamard products of varieties of matrices with fixed rank also due to their connection with problems related to algebraic statistics and quantum information. The authors introduce the notion of a Hadamard decomposition and the Hadamard rank of a matrix (multiplicative versions of the well-studied additive decomposition of tensors and tensor ranks).
One of the result in \cite{FOW} is the following:
\begin{Prop}\label{Oneto}
Let $V\subset \PP^n$ be a projective variety generated by an ideal with a minimal set of generators of the type $f_{\alpha,\beta}=X^\alpha-X^\beta$. with $\alpha,\beta \in {\mathbb{N}}^{n+1}$. Then
\[
V^{\star 2}=V.
\]
\end{Prop}
In this paper, starting from the previous result, we are interested to study which varieties are idempotent under Hadamard power. We focus, in particular on the case of hypersurfaces. 
Since, in general, the Hadamard product of two hypersurfaces $V$ and $W$ is the whole ambient space, we first study which conditions on $V$ and $W$ are required so that $V\star W$ is a hypersurface. We obtain a necessary and sufficient condition which is stated in Propositions \ref{2bin} and \ref{getbin}.
Successively we pass to study Hadamard powers of a hypersuface  $V$ obtaining a sufficient and necessary condition for which $V^{\star r}=V$. This results leads to Theorem \ref{thmidemv} which generalizes the result of Proposition \ref{Oneto} to $V^{\star r}=V$.

The paper is organized in the following way.

In Section \ref{basic} we recall  the definitions of a Hadamard product of varieties and Hadamard powers. In Section \ref{preliminary} we define the concept of a Hadamard transformation, $f^{\star P}$ of a polynomial $f$ for a point $P$ and we prove some related results leading to Theorem \ref{ThmGens} which proves the connection between the ideals of $V$ and $P\star V$.

In Section \ref{prodhyper} we study the Hamadard product of hypersurfaces, considering also the case of union of coordinate hyperplanes. Proposition \ref{2bin} shows that 
if  $C,D\subset \PP^n$ are  binomial hypersurfaces with ``similar" defining equations, then $C\star D$ is a binomial  hypersurface. Proposition \ref{getbin} show that this condition is also necessary.

The case of  Hamadard powers of hypersurfaces, treated in Section \ref{powhyper}, shows similar results for the case of  the Hamadard product of hypersurfaces. In this case the binomial equation of the hypersurface has an extra condition on the coefficients leading us to introduce the concept of ``binomial hypersurfaces of type $(t,\epsilon)$.

The generalization to any power of the result in \cite{FOW} is described in Section \ref{powvar}.

{\bf Acknowledgments}: The authors are memebrs of GNSAGA of INDAM. The second author was supported by MIUR grant Dipartimenti
di Eccellenza 2018-2022 (E11G18000350001).

\vspace{.2cm}

\section{Basic definitions}\label{basic}

We now recall  the definitions of the Hadamard product of varieties and Hadamard powers, as introduced in \cite{BCK}.
We work over the field of complex numbers $\mathbb{C}$.

\begin{Def}
Let $H_i\subset\PP^n,i=0,\ldots,n$, be the hyperplane $x_i=0$ and set
$$\Delta_i=\bigcup_{0\leq j_1<\ldots<j_{n-i}\leq n}H_{j_1}\cap\ldots\cap H_{j_{n-i}}.$$
\end{Def}

In other words, $\Delta_i$ is the $i-$dimensional variety of points having {\em at most} $i+1$ non-zero coordinates. Thus $\Delta_0$ is the set of coordinates points and $\Delta_{n-1}$ is the union of the coordinate hyperplanes. Note that elements of $\Delta_i$ have {\em at least} $n-i$ zero coordinates. We have the following chain of inclusions:
\begin{equation}\label{inclusiondelta}
\Delta_0=\{[1:0:\cdots:0],\ldots,[0:\cdots:0:1]\}\subset\Delta_1\subset\ldots\subset\Delta_{n-1}\subset\Delta_{n}=\PP^n.\end{equation}

\begin{Def} Let $p,q\in \PP^n$ be two points with coordinates $[a_0:a_1:\cdots:a_n]$ and $[b_0:b_1:\cdots:b_n]$. If $a_ib_i\not= 0$ for some $i$, their Hadamard product $p\hadamardot q$ of $p$ and $q$, is defined as
\[
p\hadamardot q=[a_0b_0:a_1b_1:\cdots:a_nb_n].
\]
If  $a_ib_i= 0$ for all $i=0,\dots, n$ then we say $p\hadamardot q$ is not defined.

Given varieties $V$ and $W$ in $\PP^n$, their Hadamard product $V\star W$ is defined as
\[
V\hadamardot W=\overline{\{p\hadamardot q: p\in V, q\in W, p\hadamardot q\mbox{ is defined} \}},
\]
where the closure is in the Zariski topology.
\end{Def}

\begin{Def} Given a positive integer $r$ and a variety $V\subset\PP^n$, the {\em r-{th} Hadamard power} of $V$  is
\[V^{\hadamardot r}= V\hadamardot V^{\hadamardot (r-1)},\]
where $V^{\hadamardot 0}=[1: \dots : 1]$.
\end{Def}

Note that $V\hadamardot W$ is a variety such that $\dim (V\hadamardot W)\leq  \dim (V)+\dim (W)$ and that $V\hadamardot W$ can be empty even if neither $V$ nor $W$ is empty. Moreover,  $\dim (V^{\hadamardot r})\leq r\dim (V) $ and the $r$-th Hadamard power cannot be empty if $V$ is not empty (\cite{BCK}).

Note that if $V$ and $W$ are irreducible, then $V \star W$ is irreducible.

For the rest of the paper we denote by $Z(f)$ (respectively $Z(f_1,\dots, f_t)$ and $Z(J)$) the zero set of a polynomial $f$ (respectively, of a set of polynomials $f_1,\dots, f_t$ and of an ideal $J$).


\section{Preliminary results}\label{preliminary}

Let $R$ be the ring $\CC[x_0, x_1, \dots, x_n]$. Given a vector of nonnegative integers $I=(i_0, \dots, i_n)$, we denote by $X^I$ the monomial $x_0^{i_0}x_1^{i_1}\cdots x_n^{i_n}$ and by $|I|=i_0+\cdots +i_n$. Similarly, if $P$ is a point of $\PP^n$ with coordinates $[p_0:p_1:\cdots :p_n]$, we denote by $P^I$ the monomial $X^I$ evaluated in $P$, that is $p_0^{i_0}p_1^{i_1}\cdots p_n^{i_n}$.
Moreover, if $P$ is a point of $\PP^n\setminus \Delta_{n-1}$ with coordinates $[p_0:p_1:\cdots :p_n]$, we denote by $\frac{1}{P}$ the point with coordinates $[\frac{1}{p_0}:\frac{1}{p_1}:\cdots :\frac{1}{p_n}]$.

\begin{Def}
Let $f\in \CC[x_0, x_1, \dots, x_n]$ be a homogenous polynomial, of degree $d$, of the form $f=\sum_{|I|=d}a_IX^I$ and consider a point $P\in \PP^n\setminus \Delta_{n-1}$. The Hadamard transformation of $f$ by $P$ is the polynomial
\begin{equation}\label{Hadtransformation}
f^{\hadamardot P}=\sum_{|I|=d}\frac{a_I}{P^I}X^I.
\end{equation}
\end{Def}

\begin{Lem}\label{polylem1}
If $P\in \PP^n\setminus \Delta_{n-1}$, one has
\[
(f^{\hadamardot P})^{\hadamardot \frac{1}{P}}=(f^{\hadamardot \frac{1}{P}})^{\hadamardot P}=f.
\]
\end{Lem}

\begin{proof}
Since $f^{\hadamardot \frac1P}=\sum_{|I|=d}a_IP^IX^I$, one has
\[
(f^{\hadamardot P})^{\hadamardot \frac{1}{P}}=(\sum_{|I|=d}\frac{a_I}{P^I}X^I)^{\hadamardot \frac{1}{P}}=\sum_{|I|=d}\frac{a_I}{P^I}P^IX^I=\sum_{|I|=d}a_IX^I=f
\]
and
\[
(f^{\hadamardot \frac{1}{P}})^{\hadamardot P} =(\sum_{|I|=d}a_IP^IX^I)^{\hadamardot P}=\sum_{|I|=d}\frac{a_IP^I}{P^I}X^I=\sum_{|I|=d}a_IX^I=f.
\]

\end{proof}

\begin{Lem}\label{polylem2}
Let $Q\in \PP^n$ and $P\in \PP^n\setminus \Delta_{n-1}$. Then $f(Q)=0$ if and only if $f^{\hadamardot P}(P\hadamardot Q)=0$.
\end{Lem}

\begin{proof}
The statement easily follows from
\begin{equation}
\begin{split}
f(Q)&=(\sum_{|I|=d}a_IQ^I)=\sum_{|I|=d}\frac{a_IP^I}{P^I}Q^I\\
&=\sum_{|I|=d}\frac{a_I}{P^I}(P^IQ^I)=\sum_{|I|=d}\frac{a_I}{P^I}(P\hadamardot Q)^I\\
&=f^{\hadamardot P}(P\hadamardot Q).
\end{split}
\end{equation}
\end{proof}

For $P\in \PP^n\setminus\Delta_{n-1}$ and $Q\in \PP^n$ denote by $\frac{Q}{P}$ the Hadamard product $Q\hadamardot \frac{1}{P}$.

\begin{Lem}\label{polylem3}
If $Q\in P\hadamardot V$ then $\frac{Q}{P}\in V$.
\end{Lem}

\begin{proof}
If $Q\in P\hadamardot V$ then $\frac{Q}{P}\in \frac{1}{P}\hadamardot P \hadamardot V$. By the associativity of the Hadamard product one has $\frac{1}{P}\hadamardot P \hadamardot V=V$, and the claim follows.
\end{proof}

\begin{Thm}\label{ThmGens}
Let $V\subset \PP^n$ be a variety and consider a point $P\in \PP^n\setminus \Delta_{n-1}$. If $f_1, \dots, f_s\subset \CC[x_0, \dots, x_n]$ is a generating set for $I(V)$, that is $I(V)=\langle f_1, \dots, f_s\rangle$, then $f_1^{\hadamardot P}, \dots, f_s^{\hadamardot P}$ is a generating set for $I(P\hadamardot V)$.

Moreover, if $f_1, \dots, f_s$ is a Groebner bases for $I(V)$, then $f_1^{\hadamardot P}, \dots, f_s^{\hadamardot P}$ is a Groebner bases for $I(P\hadamardot V)$.
\end{Thm}

\begin{proof}

We first show that the polynomials $f_1^{\hadamardot P}, \dots, f_s^{\hadamardot P}$ vanish on $P\hadamardot V$.  If $Q\in P\hadamardot V$ then $Q=P\hadamardot C$ for some $C\in V$. By proof of Lemma \ref{polylem2} one has
\[
f_i^{\hadamardot P}(Q)=f_i^{\hadamardot P}(P\hadamardot C)=f_i(C)=0, \mbox{ for all } i=1, \dots, s.
\]
Suppose now that $g\in I(P\hadamardot V)$ is different from $f_i$ for $ i=1, \dots, s$. Then $g(Q)=0$ for all $Q\in P\hadamardot V$ and, by Lemma \ref{polylem2} and \ref{polylem3}, one has
\[
g^{\hadamardot \frac{1}{P}}(\frac{Q}{P})=0 \mbox{ for all } \frac{Q}{P}\in V
\]
that is $g^{\hadamardot \frac{1}{P}}\in I(V)$. Hence, by hypothesis $g^{\hadamardot \frac{1}{P}}=\sum_i\alpha_i f_i$ and by Lemma \ref{polylem1}
\[
g=(g^{\hadamardot \frac{1}{P}})^{\hadamardot P}=(\sum_i\alpha_i f_i)^{\hadamardot P}=\sum_i(\alpha_i)^{\hadamardot P} f_i^{\hadamardot P}
\]
and thus $f_1^{\hadamardot P}, \dots, f_s^{\hadamardot P}$ is a generating set for $I(P\hadamardot V)$.

Suppose now $f_1, \dots, f_s$ is a Groebner bases for $I(V)$, that is  , there is a monomial order such that
\[
\langle LT(I(V))\rangle=\langle LT(f_1), \dots,LT(f_s)\rangle.
\]
Since $LT(f_i)$ and $LT(f_i^{\hadamardot P})$ differs only by a constant, for all $i=1, \dots, s$, one has
\[
\langle LT(f_1), \dots,LT(f_s)\rangle=\langle LT(f_1^{\hadamardot P}), \dots,LT(f_s^{\hadamardot P})\rangle.
\]
Hence it is enough to show that $\langle LT(I(V))\rangle=\langle LT(I(P\hadamardot V))\rangle$.
The monomials $LM(g)$, for all $g$ in $I(V)\setminus \{0\}$ span $\langle LM(g): g\in I(V)\setminus \{0\}\rangle$ and since $LM(g)$ and $LT(g)$ differ only by a constant one has
\[
A:=\langle LM(g): g\in I(V)\setminus \{0\}\rangle=\langle LT(g): g\in I(V)\setminus \{0\}\rangle=\langle LT(I(V))\rangle
\]
and similarly
\[
B:=\langle LM(g): g\in I(P\hadamardot V)\setminus \{0\}\rangle=\langle LT(g): g\in I(P\hadamardot V)\setminus \{0\}\rangle=\langle LT(I(P\hadamardot V))\rangle.
\]
Hence the proof is complete if we show $A=B$.

Let $h\in A$. Then there exists $g\in I(V)$ such that $h=LM(g)$. Hence $g=\sum \alpha_i f_i$ and clearly $g^{\hadamardot P}=\sum \alpha_i^{\hadamardot P}f_i^{\hadamardot P}$ and moreover $LM(g^{\hadamardot P})=LM(g)$, since $g^{\hadamardot P}$ and $g$ differ only by a constant,  from which one has $h\in B$.

Conversely if $h\in B$, then there exists $g\in I(P\hadamardot V)$ such that $h=LM(g)$. Thus $g=\sum \alpha_i f_i^{\hadamardot P}$ and, by Lemma \ref{polylem1}, $g^{\hadamardot \frac1P}=\sum \alpha_i^{\hadamardot \frac1P}f_i$ with, again, $LM(g^{\hadamardot P})=LM(g)$, from which one has $h\in A$.
\end{proof}

\begin{Rmk} It is easy to prove that if $f_1, \dots, f_s$ is a minimal Groebner bases for $I(V)$, then $f_1^{\hadamardot P}, \dots, f_s^{\hadamardot P}$ is a minimal Groebner bases for $I(P\hadamardot V)$.
\end{Rmk}


\section{Hadamard products of hypersurfaces}\label{prodhyper}

In general the Hadamard product of two  hypersurfaces is not a hypersurface but the whole ambient space.  However, the following results show that there are cases in which the Hadamard product of two hypersurfaces is a hypersurfaces.

First of all we analyze some pathological cases. 

\begin{Lem}\label{casibuffi} Let $H_i$ be any coordinate hyperplane, for $i=0,\dots, n$. Then
\begin{itemize}
\item[i)] $H_{i_1}\star \cdots \star H_{i_t}$ is the linear subspace  $Z(x_{i_j}\,;\, j=1,\dots, t)$.
\item[ii)] $H_i\star C=H_i$ for any hypersurface $C$ different from a coordinate hyperplane.
\end{itemize}
\end{Lem}

\begin{proof} For  $i)$ notice that, since the hyperplane $H_{i_j}$ is the set of points with $x_{i_j}=0$, the product $H_{i_1}\star \cdots \star H_{i_t}$ consists of all points with $x_{i_j}$=0, for $ j=1,\dots, t$, i.e. the desired linear space.
Similarly, for $ii)$, since the hyperplane $H_{i}$ is the set of points with $x_{i}=0$, the product $H_{i}\star C$ is contained in $H_{i}$=0. To see that $H_i\star C=H_i$ it is enough to apply \cite[Lemma 2.12]{BCK} which yields that the tangent space to $H_i\star C$ in $P\star Q$ is 
\[
Q\star T_P(H_i)
\]
for generic points $P$ and $Q$. Note that $T_P(H_i)=H_i$ and that, since $C$ is not contained in a coordinate hyperplane, $Q\star H_i=H_i$.

\end{proof}

Before showing the results for the Hadamard product of hypersurface, we need the following definition.

\begin{Def}
An irreducible hypersurface $C\subset \PP^n$ is called binomial hypersurface if the equation defining $C$ is of the form
\[
\alpha_1X^{I_1}-\alpha_2X^{I_2}=0.
\]
\end{Def}

\begin{Rmk}
The condition of irreducibility of $C$ in the previous definition forces $I_1$ and $I_2$ to be coprime (Lemma 4.30 in \cite{HHO}).
\end{Rmk}

\begin{Rmk}\label{puntidelta}
From the definition of the binomial hypersurface $C$, if $I_1+I_2$ has no zero entries (that is, all variables are involved in the binomial defining $C$),  then $C$ does not contain points in $\Delta_{n-1}\setminus \Delta_{n-2}$ since such points contain only one zero coordinate and so the equation on $C$ does not vanish on them.
If $I_1+I_2$ has zero entries, say, for example $(I_1+I_2)_i=0$, then $C$ contains points $P$ with $P_i=0$. However, in this case, it is still possible to apply the Hadamard transformation by such $P$ since $P_i$ will be not involved in the formula (\ref{Hadtransformation}). 

On the other side, the hypersurface $C$, given by the equation $\alpha_1X^{I_1}-\alpha_2X^{I_2}=0$, contains all subspaces  of equations
\[
\begin{cases}
x_{i_1}=0\\
x_{i_2}=0\\
\vdots\\
x_{i_t}=0\\
x_{j_1}=0\\
x_{j_2}=0\\
\vdots\\
x_{j_s}=0\\
\end{cases}
\]
where $\emptyset\not=\{i_1, i_2, \dots, i_t\}\subset I_1$ and $\emptyset\not=\{j_1, j_2, \dots, j_s\}\subset I_2$.
For any such subspace $W\subset C\subset \Delta_{n-2}$ and a point $P\in W$, the Hadamard product $P\star C$ is exactly $W$ if the only zero entries of $P$ are in coordinates $i_1, i_2, \dots, i_t, j_1, j_2, \dots, j_s$, while it is strictly contained in $W$ if $P$ as extra zero coordinates. In both case, however, one has $P\star C\subset C$.
\end{Rmk}

\begin{Prop}\label{2bin}
If $C,D\subset \PP^n$ are the following binomial hypersurfaces
\[
C=Z(\alpha_1X^{I_1}-\alpha_2X^{I_2}) \mbox{ and } D=Z(\beta_1X^{I_1}-\beta_2X^{I_2}),
\]
 then $C\star D$ is the binomial  hypersurface
 \[
C\star D=Z(\alpha_1\beta_1X^{I_1}-\alpha_2\beta_2X^{I_2}).
\]
\end{Prop}
\begin{proof}
Let $P\in C\setminus \Delta_{n-1}$. Since $P\in C$ one has
\begin{equation}\label{subs}
P^{I_1}=\frac{\alpha_2}{\alpha_1}P^{I_2}.
\end{equation}
Moreover, by Theorem \ref{ThmGens}, the equation defining $P\star D$ is
\[
 \frac{\beta_1}{P^{I_1}}X^{I_1}-\frac{\beta_2}{P^{I_2}}X^{I_2}=0.
 \]
Using (\ref{subs}) in the previous equation, we get the desired equation for $C\star D$.

If $P\in C\cap \Delta_{n-1}$ then, by Remark \ref{puntidelta}, $P$ has a zero entry in a coordinate not involved in $\alpha_1X^{I_1}-\alpha_2X^{I_2}$ or $P\in C\cap \Delta_{n-2}$. In the first case we can still apply the first part of the proof. In the second case, $P$, and hence also $P\star D$ are contained in a linear space $W\subset C,D$. Looking at the equations of $C$, $D$ and $C\star D$, again by Remark \ref{puntidelta}, one has $P\star D\in  Z(\alpha_1\beta_1X^{I_1}-\alpha_2\beta_2X^{I_2})$.

\end{proof}

\begin{Rmk}
Theorem 2.5 in \cite{BCFL1} states that if $H$ and $K$ are two hyperplanes in $\PP^n$ of equation respectively $a_ix_i+a_jx_j=0$ and $b_ix_i+b_jx_j=0$, then their Hadamard product $H\star K$ is the hyperplane of equation $a_ib_ix_i-a_jb_jx_j=0$. 

Observe that this result is a special case of Proposition \ref{2bin} just taking as $I_1$ and $I_2$ respectively the $i-$th and the $j-$th coordinate vectors and
\[
\alpha_1=a_i,\alpha_2=-a_j \mbox{ and } \beta_1=b_i,\beta_2=-b_j.
\]
\end{Rmk}
\vskip0.3cm
The following result shows that also the converse of Proposition \ref{2bin} is true.

\begin{Prop}\label{getbin}
Let $C,D$ be irreducible  hypersurfaces not contained in $\Delta_{n-1}$. If $C\star D$ is a hypersurface, then $C$ and $D$ are binomials, such that
\[
C=Z(\alpha_1X^{I_1}-\alpha_2X^{I_2}) \mbox{ and } D=Z(\beta_1X^{I_1}-\beta_2X^{I_2}).
\]
Moreover, $C\star D=Z(\alpha_1\beta_1X^{I_1}-\alpha_2\beta_2X^{I_2})$.

\end{Prop}

\begin{proof}
Let $E=C\star D$ and suppose that
\[
\begin{array}{rcl}
C&=&Z(\alpha_1X^{I_1}+\alpha_2X^{I_2}+\cdots + \alpha_sX^{I_s}),\\
D&=&Z(\beta_1X^{J_1}+\beta_2X^{J_2}+\cdots + \beta_rX^{J_r}),\\
E&=&Z(\gamma_1X^{K_1}+\gamma_2X^{K_2}+\cdots + \gamma_tX^{K_t}).
\end{array}
\]
By hypothesis $C\star Q\subseteq E$ for all $Q\in D$. But since $D$ is not contained in $\Delta_{n-1}$, for all points $Q\in D\setminus \Delta_{n-1}$ one has that $C\star Q$ is a hypersurface, hence  $C\star Q=E$. Thus, by Theorem \ref{ThmGens}, 
\[
Z\left(\alpha_1X^{I_1}+\alpha_2X^{I_2}+\cdots + \alpha_sX^{I_s})^{\star Q}\right)=E
\]
and hence $s=t$ and, possibly after relabelling, $I_i=K_i$ for $i=1, \dots, t$.

The same argument applied to $P\star D=E$, for $P\in C$ gives  $r=t$ and, possibly after relabelling, $J_i=K_i$ for $i=1, \dots, t$.

Consider now a point $P\in C\setminus \Delta_{n-1}$. Again by Theorem \ref{ThmGens} one has that the equation defining $P\star D$ is
\[
\frac{\beta_1}{P^{I_1}}X^{I_1}+\frac{\beta_2}{P^{I_2}}X^{I_2}+\cdots + \frac{\beta_s}{P^{I_s}}X^{I_s}=0
\]
and since $P\star D=E$ one has
\[
\frac{\beta_i}{P^{I_i}}=\lambda \gamma_i, \mbox{ for } i=1, \dots, s
\]
for a suitable choice of $\lambda$.
Hence the points of $C$ satisfy 
\[
\frac{\gamma_1}{\beta_1}X^{I_1}-\frac{\gamma_i}{\beta_i}X^{I_i}=0, \mbox{ for } i=2, \dots, s.
\]
Thus the ideal
\[
\langle \frac{\gamma_1}{\beta_1}X^{I_1}-\frac{\gamma_2}{\beta_2}X^{I_2}, \dots, \frac{\gamma_1}{\beta_1}X^{I_1}-\frac{\gamma_s}{\beta_s}X^{I_s} \rangle
\]
is contained in the ideal $\langle \alpha_1X^{I_1}+\alpha_2X^{I_2}+\cdots + \alpha_sX^{I_s} \rangle$, but this is possible only if $s=2$ and
\begin{equation}\label{coeffs}
\frac{\gamma_1}{\beta_1}=\rho\alpha_1, \, \frac{\gamma_2}{\beta_2}=-\rho\alpha_2.
\end{equation}
It follows, possibly changing the signs, that
\[
\begin{array}{rcl}
C&=&Z(\alpha_1X^{I_1}-\alpha_2X^{I_2})\\
D&=&Z(\beta_1X^{I_1}-\beta_2X^{I_2})
\end{array}
\]
and, using (\ref{coeffs}) that
\[
E=Z(\alpha_1\beta_1X^{I_1}-\alpha_2\beta_2X^{I_2}).
\]

\end{proof}

\section{Hadamard powers of  hypersurfaces}\label{powhyper}

We study now which hypersurfaces are idempotent under Hadamard powers.
The first easy case to study concerns union, eventually with multiplicities, of coordinate hyperplanes $H_i$.

\begin{Lem}
Let $C\subset \PP^n$ be the  reducible  hypersurface  $C=H_{i_0}\cup H_{i_1}\cup \cdots \cup H_{i_s}$, union of distinct coordinate hyperplanes. Then $C^{\hadamardot t}=C$ for all $t\geq 1$.
\end{Lem}

\begin{proof}
Since $H_i^{\star t}=H_i$, for every $t$ and $H_i\star H_j=H_i\cap H_j$ (by part i of Lemma \ref{casibuffi}), expanding all terms in $(H_{i_0}\cup H_{i_1}\cup \cdots \cup H_{i_s})^{\star t}$ we get the desired equality.
\end{proof}

\begin{Def} Let $C\subset \PP^n$ be a binomial hypersurface. We say that $C$  is of type  $(t,\epsilon)$ if 
\[
\alpha_1=1 \mbox{ and } \alpha_2=\xi^{\epsilon} 
\]
where   $\xi$ is a primitive $(t-1)-$th root of unity and $1\leq \epsilon \leq t-1$.
\end{Def}

\begin{Prop}\label{propidem1}
If $C\subset \PP^n$ be a binomial hypersurface of  type $(t,\epsilon)$, then $C^{\hadamardot t}=C$. Moreover, if $gcd(t-1,\epsilon)=1$ then $t$ is the minimal exponent for which the previous equality holds.
\end{Prop}

\begin{proof}
By hypothesis the equation  of $C$ is $F(x_0,\dots, x_n)= X^{I_1}-\xi^{\epsilon} X^{I_2}$. Consider $t-1$ points $P_i\in C\setminus \Delta_{n-1}$. Each point satisfies $P_i^{I_1}=\xi^{\epsilon} P_i^{I_2}$ and
\begin{equation}\label{prodeqgen}
\begin{split}
(P_1\hadamardot \cdots \hadamardot P_{t-1})^{I_1}&=P_1^{I_1}\cdots P_{t-1}^{I_1}\\
&=\xi^{\epsilon}P_1^{I_2}\cdots \xi^{\epsilon}P_{t-1}^{I_2} \\
&=(\xi^{\epsilon})^{t-1}P_1^{I_2}\cdots P_{t-1}^{I_2} \\
&=P_1^{I_2}\cdots P_{t-1}^{I_2}\\
&=(P_1\hadamardot \cdots \hadamardot P_{t-1})^{I_2}.
\end{split}
\end{equation}

By Theorem \ref{ThmGens}, the equation for $P_1\hadamardot \cdots \hadamardot P_{t-1} \hadamardot C$ is
\[
F^{\hadamardot (P_1\hadamardot \cdots \hadamardot P_{t-1} )}(x_0,\dots,x_n)=\frac{ X^{I_1}}{(P_1\hadamardot \cdots \hadamardot P_{t-1})^{I_1}}- \frac{\xi^{\epsilon}X^{I_2}}{(P_1\hadamardot \cdots \hadamardot P_{t-1})^{I_2}}
\]
and, by (\ref{prodeqgen}), one has  that $F^{\hadamardot (P_1\hadamardot \cdots \hadamardot P_{t-1} )}(x_0,\dots,x_n)$ and $F(x_0,\dots,x_n)$ define the same hypersurface, hence $P_1\hadamardot \cdots \hadamardot P_{t-1} \hadamardot C=C$.
Since this holds for every choice of $P_1,\dots, P_{t-1}\in C\setminus \Delta_{n-2}$ and, by Remark \ref{puntidelta}, $P_1\hadamardot \cdots \hadamardot P_{t-1} \hadamardot C\subset  C$ if $P_1\hadamardot \cdots \hadamardot P_{t-1}  \in C\cap \Delta_{n-2}$, one has $C^{\hadamardot t}=C$.

Suppose now that $gcd(t-1,\epsilon)=\alpha>1$. Since $\xi^\epsilon=e^{\frac{2\pi \epsilon i}{t-1}}$ and assuming that $t-1=\alpha(r-1)$ and $\epsilon=\alpha \theta$ one has $\xi^\epsilon=e^{\frac{2\pi \alpha\theta i}{\alpha(r-1)}}=e^{\frac{2\pi \theta i}{r-1}}$. Hence $\xi^\epsilon$ is the power of a   primitive $(r-1)-$th root of unity and, by the first part of the proof, we get $C^{\star r}=C$. 

\end{proof}

The following results prove that being  binomial hypersurface of type $(t,\epsilon)$ is also a necessary condition.
\begin{Prop}
Let $C\subset \PP^n$ be an irreducible hypersurface not contained in $\Delta_{n-1}$. If $C^{\hadamardot t}=C$ then $C$ is a binomial hypersurface of type $(t,\epsilon)$.
In particular, if $gdc(t-1, \epsilon)=1$ then $t$ is the minimal exponent for which $C^{\hadamardot t}=C$ holds.
\end{Prop}
\begin{proof}
Let $C$ be a hypersurface satisfying  $C^{\hadamardot t}=C$. Hence, by dimensional reasons, each successive Hadamard product in 
\[
C\star(C\star (C\star(\cdots (C\star C))))
\]
 is a hypersurface. In particular, by Proposition \ref{getbin}, applied to $C\star C$, one has that $C$ must be  binomial of equation $\alpha_1X^{I_1}-\alpha_2X^{I_2}=0$. Then by Corollary \ref{getbin} we get that the equation of $C^{\star t}$ is
 $\alpha_1^tX^{I_1}-\alpha_2^tX^{I_2}=0$. 

Finally, the condition $C^{\star t}=C$ implies that $\alpha_i^t=\lambda \alpha_i$,  for $ i=1,2$ and a suitable nonzero $\lambda$. Thus we get
\[
\frac{\alpha_2}{\alpha_1}\left\lbrack\left(\frac{\alpha_2}{\alpha_1}\right)^{t-1}-1\right\rbrack=0.
\]
Excluding the trivial solution $\alpha_2=0$, we get that $\frac{\alpha_2}{\alpha_1}$ is a $(t-1)-$th root of unity, showing that $C$ is binomial of type $(t,\epsilon)$.

If $gdc(t-1, \epsilon)=1$, then,  by Proposition \ref{propidem1}, we get the minimality of $t$.

\end{proof}

We end this section with an analogue results for reducible hypersurfaces.

\begin{Prop}
Let $C_1\cup\cdots \cup C_s$ be a reducible hypersurface such that  $C_i\not\subset \Delta_{n-1}$ for each component $C_i$, $i=1, \dots,s$. Then $(C_1\cup\cdots \cup C_s)^{\star r}=C_1\cup\cdots \cup C_s$ if and only if
\begin{itemize}
\item[1.] each component $C_i$ is a binomial hypersurfaces of type $(t,\epsilon_i)$ and equation $X^{I_1}-\xi^{\epsilon_i}X^{I_2}$, for $i=1,\dots, s$; 
\item[2.] all products $\xi^{d_1\epsilon_1}\cdots \xi^{d_s \epsilon_s}$, with $d_1+d_2+\cdots +d_s=r$, is equal to $\xi^{\epsilon_j}$ for some $j=1,\dots, s$.
\end{itemize}
\end{Prop}

\begin{proof}
Assume 1. and 2. holds. Any terms in $(C_1\cup\cdots \cup C_s)^{\star r}$ has the form $C_1^{\star d_1}\star C_2^{\star d_2}\star \cdots \star C_s^{\star d_s}$, with $d_1+d_2+\cdots +d_s=r$. Condition 2, together with Proposition \ref{2bin}, applied to the binomial hypersurfaces $C_i$'s, of type $(t,\epsilon_i)$, implies that $C_1^{\star d_1}\star C_2^{\star d_2}\star \cdots \star C_s^{\star d_s}=C_j$ for some $j=1,\dots, s$.

For the other direction, the condition $(C_1\cup\cdots \cup C_s)^{\star r}=C_1\cup\cdots \cup C_s$ implies that $C_1^{\star d_1}\star C_2^{\star d_2}\star \cdots \star C_s^{\star d_s}$, with $d_1+d_2+\cdots +d_s=r$, is equal to $C_j$ for some $j=0, \dots, s$. In particular, if we consider $d_i=r$ we get $C_i^{\star r}=C_j$ for some $j=1, \dots, s$ and for all $i=1, \dots s$. Repeatedly applying Proposition \ref{getbin} (starting from $C_i\star C_i^{\star r-1}$) it follows that $C_i$ is a binomial hypersurface for $i=1,\dots,s$ and, moreover, that all these components are defined by the same binomial equation, with different coefficients.
Let assume that $C_i=Z(X^{I_1}-\alpha_iX^{I_2})$, for $i=1,\dots,s$. Hence, by the previous consideration we have
\begin{equation}\label{prodalpha}
\begin{array}{ll}
\alpha_1^{d_1}\cdots\alpha_s^{d_s}=\alpha_j &\mbox{ for some } j=1,\dots s,\\
&\mbox{ with } d_1+\cdots + d_s=r.
\end{array}
\end{equation}

We claim that $|\alpha_i|=1$ for $i=1,\dots, s$. To this aim, assume, eventually reordering the components, that
\[
|\alpha_1|\leq |\alpha_2| \leq \cdots \leq |\alpha_s|.
\]
If there exists a index $j$ such that $|\alpha_j|>1$, then $|\alpha_s|>1$. Taking the Hadamard power $C_s^{\star r}$ we get, by Proposition \ref{2bin}, a hypersurface defined by the  polynomial $X^{I_1}-\alpha_s^rX^{I_2}$. Since $|\alpha_s^r|>|\alpha_s|$, condition (\ref{prodalpha}) is not satisfied and we get a contradiction.
Similarly, if there exists a index $j$ such that $|\alpha_j|<1$, then $|\alpha_1|<1$. Taking the the Hadamard power $C_1^{\star 1}$ we get a hypersurface defined by the  polynomial $X^{I_1}-\alpha_1^rX^{I_2}$. Since $|\alpha_1^r|<|\alpha_1|$, condition (\ref{prodalpha}) is not satisfied and, again, we get a contradiction.

Since $\alpha_i$ lies in the unitary circle, we can write  $\alpha_i=\eta_i^{\rho_i}$ where  $\eta_i$ is a $t_i-$th root of unity, and $1\leq \rho_i \leq t_i$, for $i=1, \dots,s$.

Let $t-1=lcm(t_1, t_2, \dots, t_s)$, then we can see each $\alpha_i$ as a $(t-1)-$th root of unity and  write $\alpha_i=\xi^{\epsilon_i}$, where $\xi$ is a primitive $(t-1)$-th root of unity and $1\leq \epsilon_i\leq t-1$, for $i=1, \dots ,s$. Hence each component $C_i$ is a binomial hypersurfaces,  of type $(t,\epsilon_i)$, $i=1,\dots, s$, and condition 1. is verified.

Condition 2. now follows directly substituting $\alpha_i=\xi^{\epsilon_i}$ in condition (\ref{prodalpha}).
\end{proof}

\begin{Ex}
Let $\xi$ be the $6-$th root of unity $e^{\frac{2\pi i}{6}}$ and define the hypersurfaces
\[ 
\mathcal{C}_j=Z(h_j)
\]
where $h_j=X^{I_1}-\xi^jX^{I_2}$, $j=1,\dots 6$, and  $X^{I_1}$ and $X^{I_2}$ are coprime monomials of degree $d$.

According to Proposition \ref{propidem1} we have the following Hadamard multiplication table

\[
\begin{array}{c|c|c|c|c|c|c|}
& \mathcal{C}_1 & \mathcal{C}_2 & \mathcal{C}_3 & \mathcal{C}_4 & \mathcal{C}_5 & \mathcal{C}_6 \\
\hline
\mathcal{C}_1 & \mathcal{C}_2 & \mathcal{C}_3 & \mathcal{C}_4 & \mathcal{C}_5 & \mathcal{C}_6 & \mathcal{C}_1 \\
\hline
\mathcal{C}_2 & \mathcal{C}_3 & \mathcal{C}_4 & \mathcal{C}_5 & \mathcal{C}_6 & \mathcal{C}_1 & \mathcal{C}_2 \\
\hline
\mathcal{C}_3 & \mathcal{C}_4 & \mathcal{C}_5 & \mathcal{C}_6 & \mathcal{C}_1 & \mathcal{C}_2 & \mathcal{C}_3 \\
\hline
\mathcal{C}_4 & \mathcal{C}_5 & \mathcal{C}_6 & \mathcal{C}_1 & \mathcal{C}_2 & \mathcal{C}_3 & \mathcal{C}_4 \\
\hline
\mathcal{C}_5 & \mathcal{C}_6 & \mathcal{C}_1 & \mathcal{C}_2 & \mathcal{C}_3 & \mathcal{C}_4 & \mathcal{C}_5 \\
\hline
\mathcal{C}_6 & \mathcal{C}_1 & \mathcal{C}_2 & \mathcal{C}_3 & \mathcal{C}_4 & \mathcal{C}_5 & \mathcal{C}_6 \\
\hline
\end{array}
\]
\end{Ex}

Consider $C=\mathcal{C}_1\cup\mathcal{C}_3\cup\mathcal{C}_5$.
If we compute $C^{\star2}$, using the previous table, we get 
\[
(\mathcal{C}_1\cup\mathcal{C}_3\cup\mathcal{C}_5)^{\star 2}=\mathcal{C}_2\cup\mathcal{C}_4\cup\mathcal{C}_6.
\]
Then, computing $(\mathcal{C}_1\cup\mathcal{C}_3\cup\mathcal{C}_5)^{\star 3}$ we get

\begin{equation*}
\begin{split}
    (\mathcal{C}_1\cup\mathcal{C}_3\cup\mathcal{C}_5)^{\star 3}&=(\mathcal{C}_2\cup\mathcal{C}_4\cup\mathcal{C}_6)\star (\mathcal{C}_1\cup\mathcal{C}_3\cup\mathcal{C}_5)\\
    &=\mathcal{C}_1\cup\mathcal{C}_3\cup\mathcal{C}_5.
     \end{split}
\end{equation*}
Notice that Conditions 1. and 2. are satisfied for $\mathcal{C}_1,\mathcal{C}_3,\mathcal{C}_5$ and $r=3$.

\section{On Hadamard powers of varieties}\label{powvar}

The result presentend in this section, generalizes the sufficient condition to have $V^{\star 2}=V$, in \cite{FOW} and stated in Proposition \ref{Oneto} in the introduction of this paper, to the case of higher powers.

Similarly to the case of hypersurfaces, we define a specific class of varieties.

\begin{Def} Let $C\subset \PP^n$ be a binomial variety. We say that $C$  is of type  $[(t_{1},\epsilon_{1}),\dots, (t_{s},\epsilon_{s})]$ if the ideal of $C$ is generated by
\[
X^{I_{1,1}}-\xi_{1}^{\epsilon_{1}}X^{I_{1,2}}, \dots, X^{I_{s,1}}-\xi_{s}^{\epsilon_{s}}X^{I_{s,2}},
\]
where  $\xi_{i}$ is a primitive $(t_{i}-1)-$th root of unity and $1\leq \epsilon_{i} \leq t_{i}-1$ for $i=1,\dots, s$.
\end{Def}

\begin{Thm}\label{thmidemv}
Let $C\subset \PP^n$ be a binomial variety of  type $[(t_{1},\epsilon_{1}),\dots, (t_{s},\epsilon_{s})]$ with  
\begin{equation}\label{formulona}
lcm\left(\frac{t_{1}}{gcd(t_{1},\epsilon_{1})},\frac{t_{2}}{gcd(t_{2},\epsilon_{2})}\dots, \frac{t_{s}}{gcd(t_{s},\epsilon_{s})}\right)=t-1.
\end{equation}
Then $C^{\hadamardot t}=C$  and $t$ is the minimal exponent for which the previous equality holds.
\end{Thm}

\begin{proof}
By hypothesis the ideal of  $C$  is generated by $F_j(x_0,\dots, x_n)=X^{I_{j,1}}-\xi_{j}^{\epsilon_{j}}X^{I_{j,2}}$ for $j=1, \dots, s$. Consider $t-1$ points $P_i\in C\setminus \Delta_{n-1}$. Each point satisfies $P_i^{I_{j,1}}=\xi_{j}^{\epsilon_{j}} P_i^{I_{j,2}}$ and
\begin{equation}\label{prodeqgenv}
\begin{split}
(P_1\hadamardot \cdots \hadamardot P_{t-1})^{I_{j,1}}&=P_1^{I_{j,1}}\cdots P_{t-1}^{I_{j,1}}\\
&=(\xi_{j}^{\epsilon_{j}})^{t-1}P_1^{I_{j,2}}\cdots P_{t-1}^{I_{j,2}} \\
&=P_1^{I_{j,2}}\cdots P_{t-1}^{I_{j,2}}\\
&=(P_1\hadamardot \cdots \hadamardot P_{t-1})^{I_{j,2}}
\end{split}
\end{equation}
where the third equality follows from the fact that $t-1$ is the least integer such that $1=(\xi_{j}^{\epsilon_{j}})^{t-1}$, for all $j=1, \dots s$ since (\ref{formulona}) holds.

By Theorem \ref{ThmGens}, the equations for $P_1\hadamardot \cdots \hadamardot P_{t-1} \hadamardot C$ are
\[
F_j^{\hadamardot (P_1\hadamardot \cdots \hadamardot P_{t-1} )}(x_0,\dots,x_n)=\frac{ X^{I_{j,1}}}{(P_1\hadamardot \cdots \hadamardot P_{t-1})^{I_{j,1}}}- \frac{\xi_{j}^{\epsilon_{j}}X^{I_{j,2}}}{(P_1\hadamardot \cdots \hadamardot P_{t-1})^{I_{j,2}}}
\]
and, by (\ref{prodeqgenv}), one has  that $F_j^{\hadamardot (P_1\hadamardot \cdots \hadamardot P_{t-1} )}(x_0,\dots,x_n)$ and $F_j(x_0,\dots,x_n)$ are equal up to a multiplicative constant, hence $P_1\hadamardot \cdots \hadamardot P_{t-1} \hadamardot C=C$.
Since this holds for every choice of $P_1,\dots, P_{t-1}\in C\setminus \Delta_{n-2}$ and, by Remark \ref{puntidelta}, $P_1\hadamardot \cdots \hadamardot P_{t-1} \hadamardot C\subset  C$ if $P_1\hadamardot \cdots \hadamardot P_{t-1}  \in C\cap \Delta_{n-2}$, one has $C^{\hadamardot t}=C$.
\end{proof}

\end{document}